\documentclass[11pt,a4paper]{amsart}
\usepackage[english]{babel}
\usepackage[utf8]{inputenc} 
\usepackage[T1]{fontenc}
\usepackage{verbatim}
\usepackage{palatino}
\usepackage{amsmath}
\usepackage{mathabx}
\usepackage{amsbsy}
\usepackage{amssymb}
\usepackage{amsthm}
\usepackage{amsfonts}
\usepackage{graphicx}
\usepackage{esint}
\usepackage{color}
\usepackage{mathtools}
\usepackage{overpic}
\usepackage{mathrsfs}
\usepackage{tabularx}
\usepackage{datetime}
\renewcommand{\today}{\the\day\ \shortmonthname[\month] \the\year}

\usepackage[colorlinks = true, citecolor = black]{hyperref}
\author{William O'Regan}
\title{Covering sponges with tubes}
\email{woregan@math.ubc.ca}
\date{\today}
\subjclass[2010]{05B99, 28A78, 28A80}
\keywords{tube-null, orthogonal projection}
\thanks{This work was completed while the author was supported by the EPSRC via the project \emph{Ergodic and combinatorial methods in fractal geometry}, project ref.~2443767.}

\newcommand{\E}{\mathrm{H}}

\newcommand{\p}{\mathrm{p}}
\newcommand{\calV}{\mathcal{V}}

\newcommand{\Leb}{\mathcal{L}^1|_{[0,1)}}

\newcommand{\M}{\calM}
\newcommand{\Ec}{\calE}
\newcommand{\om}{\omega}

\newcommand{\R}{\mathbb{R}}
\newcommand{\pr}{\mathbb{P}}

\newcommand{\N}{\mathbb{N}}

\newcommand{\Z}{\mathbb{Z}}

\newcommand{\Ga}{\Gamma}

\newcommand{\calL}{\mathcal{L}}
\newcommand{\Le}{\mathcal{L}}

\newcommand{\calE}{\mathcal{E}}
\newcommand{\calF}{\mathcal{F}}
\newcommand{\F}{\mathcal{F}}

\newcommand{\calM}{\mathcal{M}}

\newcommand{\spt}{\operatorname{spt}}

\newcommand{\dimh}{\dim_{\rm{H}}}

\newcommand{\V}{\mathcal{V}}

\numberwithin{equation}{section}
\theoremstyle{plain}
\newtheorem{thm}[equation]{Theorem}

\newtheorem{lemma}[equation]{Lemma}

\newtheorem{cor}[equation]{Corollary}

\newtheorem{prop}[equation]{Proposition}

\newtheorem{claim}[equation]{Claim}

\newtheorem{problem}[equation]{Problem}

\theoremstyle{definition}

\newtheorem{definition}[equation]{Definition}

\theoremstyle{remark}

\newlength\tindent

\renewcommand{\indent}{\hspace*{\tindent}}
\newcommand{\nref}[1]{(\hyperref[#1]{#1})}
\begin{document}
\begin{abstract}
    The aim of this note is to give a short proof of a result of Py\"or\"al\"a--Shmerkin--Suomala--Wu; the Sierpi\'nski carpet, and generalisations, are tube-null; they can be covered with tubes of arbitrarily small total width. We remark that a more general class of sponge-like sets satisfy this property. For a given $\epsilon > 0$ the proof is able to give an explicit description of the tubes for which the total width is less than $\epsilon.$
\end{abstract}
\maketitle
\label{chapter.tubenull}
\section{Introduction}
We call a closed $\delta/2$-neighbourhood of a line in Euclidean space a \textit{tube} of width $\delta.$ We say that a subset of Euclidean space is \textit{tube-null} if it can be covered by tubes of arbitrarily small total width. A question that attracts much attention in harmonic analysis is which functions is one able to recover the function from its Fourier transform. 
\subsection{The localisation problem}
\begin{definition}
    Let $f$ be a function on $\R^d.$ For $R > 0$ define the \textit{spherical mean of radius} $R$ by 
$$S_Rf(x) = \int_{|\xi| < R}\hat f(\xi) e^{2\pi i x \cdot \xi} d\xi.$$
\end{definition}
One of the most interesting and difficult problems in harmonic analysis is determining whether we can recover the values of every $f \in L^2(\R^d)$ from the pointwise limit of its spherical means $S_Rf.$ That is,
\begin{problem}
    Is it true that for all $f \in L^2(\R^d)$ we have
    $$\lim_{R\rightarrow \infty} S_Rf(x) = f(x) \text{ a.e.}$$
\end{problem} 
For $d = 1,$ the result is true; this is an extension to the real line of a result of Carleson \cite{carleson} \cite{kenig}. The problem is open for $d \geq 2.$ Carbery, Soria, and Vargas \cite{car2} showed that if $K \subset B(0,1)$ is so-called `tube-null', then there exists a function $f \in L^2(\R^d)$ which is identically zero on $B(0,1)$ but $S_Rf(x)$ fails to converge for every $x \in K.$ It is an open problem to characterise all such sets of divergence; in particular, it is not known if such a set is tube-null. If the assumption that $\spt f \subset \R^d\setminus B(0,1)$ is dropped, then it is not even known if the divergence set is Lebesgue null.
\subsection{The definition of a tube-null set}
\begin{definition}
We call a \textit{tube} $T$ of width $w = w(T)> 0$ the closed $w/2$-neighbourhood of some line in $\mathbb{R}^d,$ where $d \geq 2$ is an integer.
\end{definition}
\begin{definition} \label{tube-null}
 A set $K \subset \mathbb{R}^d$ is called \textit{tube-null} if for every $\epsilon > 0$ there exists a countable family of tubes $\{T_i\}_{i \in \N}$ such that 
  $$K \subset \bigcup_{i\in \N}T_i.$$ and 
 $$\sum_{i\in \N}w(T_i)^{d-1} < \epsilon.$$  
\end{definition}
An easy example of a tube-null set is the following: Let $C$ be the middle-1/3 Cantor set and consider $C \times \mathbb{R}.$ 
Let $\epsilon > 0.$ Since $\calL^1(C) = 0$ we can find a cover $\{U_i\}_{i \in \N}$ of $C$ by closed intervals such that $\sum|U_i| < \epsilon.$ Then consider the tubes $\{T_i\}_{i \in \N}$ where $T_i = U_i \times \mathbb{R}.$
\indent The notion of tube-nullity is also very natural from the point of view of geometric measure theory, and along with several variants, has been considered in many works. See, for example, \cite{car1}, \cite{chentubes}, \cite{csowise}, \cite{har}, \cite{orptube}, \cite{pyo}, \cite{shmsuotube}, \cite{shmsuomart}. It is often difficult to verify whether a given set is tube-null or not. Often the connection between tube-nullity and geometric measure theory arises from orthogonal projections. This can be seen below. 
\begin{prop}\label{hyper}
Let $K \subset \mathbb{R}^d.$ Suppose there exists a countable decomposition 
$$K = \bigcup_{n=1}^\infty K_n,$$
    a countable family of $d-1$-dimensional hyperplanes $\{V_n\}_{n \in \N},$ $V_n \in G(d,d-1),$  and orthogonal projections $P_{V_n}:\mathbb{R}^d \rightarrow V_n$ with $\calL^{d-1}(P_{V_n}(K_n)) = 0.$ Then $K$ is tube-null.
\begin{proof}
Let $\epsilon > 0.$ Since for each $n \in \N$ we have $\calL^{d-1}(P_n(K_n)) = 0$ we can find a covering of $P_{V_n}(K_n)$ by $d-1$ dimensional closed balls $\{B_{n,i}\}_{i \in \N}$ with $\sum_{i \in \N} |B_{n,i}|^{d-1} < \epsilon/2^n.$ Note here $|\cdot|$ denotes diameter. Let $\{T_{n,i}\}_{i,n \in \N}$ be the collection of tubes defined by $T_{n,i} = P_n^{-1}B_{n,i}.$ Since 
$$K_n \subset \bigcup_{i\in \N} T_{n,i}$$
and  
$$\sum_{i\in \N} w(T_{n,i})^{d-1} < \epsilon/2^n,$$
we therefore have
$$K \subset  \bigcup_{n,i\in \N} T_{n,i},$$
and  
$$\sum_{n,i \in \N} w(T_{n,i})^{d-1} < \epsilon.$$
\end{proof}
\end{prop}
\begin{prop}\label{prop.nottubenull}
Let $K \subset \R^d$ and suppose that $K$ supports a non-zero measure for which all of its orthogonal projections to $d-1$-dimensional planes are absolutely continuous with respect to $\calL^{d-1},$ each with a density which is uniformly bounded, then $K$ is not tube-null. 
\begin{proof}
 Let $\{T_i\}_{i \in \N}$ be a cover of $K$ with tubes. For each tube $T_i$ we have $\mu(T_i) \leq Cw(T_i)^{d-1},$ for some uniform $C > 0.$ Therefore
    $$0 < \mu(K) = \mu\bigg(\bigcup_{i\in \N} (K \cap T_i)\bigg) \leq \sum_{i\in \N}\mu(T_i) \leq C\sum_{i\in \N} w(T_i)^{d-1}.$$
\end{proof}
\end{prop}
\indent Since orthogonal projections are Lipschitz mappings they cannot increase Hausdorff dimension, and so sets with Hausdorff dimension strictly less than $d-1$ are tube-null. Using the Besicovitch--Federer projection theorem, Carbery, Soria, and Vargas \cite{car2} showed that sets with $\sigma$-finite $(d-1)$-dimensional Hausdorff measure are tube-null. Given this, the question of tube-nullity is interesting for sets of Hausdorff dimension at least $d-1.$ Using a random construction, Shermkin and Suomala \cite{shmsuotube} showed that there are sets of any Hausdorff dimension between $d-1$ and $d$ inclusive, that are not tube-null, and excluding the case of Hausdorff dimension $d-1,$ can be taken to be Ahlfors-regular. The construction in $\R^2$ is roughly as follows: Start with the unit square and divide into four pieces and either keep all four squares or just one of them, where this choice is done via some appropriate probability distribution. Now, for each surviving square, divide into 4 new squares. Either keep what we have already, or for each surviving square in the previous step, keep only one of the new squares. Continue ad infinitum. For the details see \cite{shmsuotube}.
 
\indent Carbery, Soria, and Vargas had shown this before for $s \in (3/2,2]$ by giving explicit examples of rotationally invariant Cantor sets \cite{car2}. They also gave examples of sets which are tube-null, for $s \in (1,3/2)$ \cite{car2}.

\indent Tube-nullity itself does not impose a bound on the Hausdorff dimension: Let $C\subset \R^d $ be a set of Hausdorff dimension $d-1$ but $\calL^{d-1}(C) = 0$ (for example, a Cantor type construction). Then $C \times [0,1]$ has Hausdorff dimension $d$ but is tube-null. Heuristically, we should expect sets of larger Hausdorff dimension to be less likely to be tube-null.
\subsection{Examples of tube-null sets}
We have the result of Harangi \cite{har}. In the plane, let $R$ be the rotation by $60^\circ$ and $R'$ the rotation by $-60^\circ.$ Define $f_1,f_2,f_3,f_4:\R^2 \rightarrow \R$ by the maps
$$f_1(x,y) = (x,y)/3; \qquad f_2(x,y) = R(x,y)/3 + (1/3,0);$$
$$f_3(x,y) = R'(x,y)/3 + (2/3,0); \qquad f_4(x,y) = (x,y)/3 + (2/3,0).$$
Let $\calF$ be the IFS consisting of these maps and let $K$ be the attractor. We refer to $K$ as the \textit{Koch curve}
\begin{thm}[Theorem 1.1, \cite{har}]
    The Koch curve is tube-null. 
\end{thm}
We also have a large class of examples given by Py\"or\"al\"a, Shmerkin, Suomala, and Wu.  Fix an integer $N \geq 2$ and let $\Ga \subset \{0,\dots ,N-1\}^d$ such that $|\Ga| < N^d.$ Consider the homogeneous IFS on $[0,1]^d$ defined by 
$$\F = \Big\{f_i(x) = \frac{x}{N} + \frac{i}{N}\Big\}_{i \in \Gamma}.$$
Let $K$ be the attractor of $\F,$ i.e., the unique non-empty compact set $K$ such that
\begin{equation}\label{eq.attractor}
    K = \bigcup_{i \in \Ga}f_i(K).
\end{equation}
\begin{thm}[Theorem 1.1, \cite{pyo}]\label{tubenull.thm}\label{main}
    The set $K$ is tube-null.
\end{thm}
Combining with what is known, this shows that for every $s \in [d-1,d]$ there exists a set $K$ with $\dimh K = s$ for which $K$ is tube-null. Therefore, we have sets which are tube-null and sets which are not tube-null at every Hausdorff dimension $s \in [d-1,d].$ They also showed the following.
\begin{definition}
Define the map $T:[0,1)^d \rightarrow [0,1)^d$ by $T(x) = Nx\mod1.$ We call the map $T$  the $\times N$-\textit{map}.
\end{definition}
\begin{cor}[Theorem 1.1, \cite{pyo}]
    Let $L \subsetneq [0,1]^d$ be a closed $\times N$ invariant set. Then $L$ is tube-null. 
    \end{cor}
    We include the short proof.
    \begin{proof}
Given any closed $T$-invariant $L \subsetneq [0,1]^d$ we can find $q$ such that not all words in $(\{0,\dots ,N-1\}^d)^q$ appear in $L$ under the natural symbolic coding. Let $K$ be the self-similar set as above, corresponding to $N^q$ and $\Ga$ in correspondence with the words of length $q$ that appear in $L,$ then $L \subset K \subsetneq [0,1]^d.$ 
\end{proof}

Below, we give a short proof of Theorem \ref{tubenull.thm}. The proof is essentially a shorter discretised version of the proof given in \cite{pyo}, that is, we consider the cylinder sets at level $n,$ and find a fairly explicit covering of these sets by tubes. 
\subsection{Preliminaries and notation}\label{ifs}
Let $A \subset \R^d$ be closed. We call a map $f:A \rightarrow A$ a \textit{contraction} if there exists $0 < r < 1$ such that
$$|f(x) - f(y)| \leq r|x-y| \text{ for all } x,y \in A.$$
If we have equality in the above, then we call $f$ a \textit{contracting similarity}. We call a finite family $\F$ of contractions an \textit{iterated function system} or IFS.   If a set $A \neq \emptyset$ is such that
$$A = \bigcup_{f \in \F}f(A)$$
then we call $A$ the \textit{attractor} of $\F.$ An IFS $\F$ satisfies the \textit{open set condition} if there exists a non-empty bounded open set $V$ such that
$$\bigcup_{f \in \F}f(V) \subset V$$
with the union disjoint. It is well-known that every IFS has a unique attractor. We call attractors of IFSs consisting of contracting similarities \textit{self-similar}.
Often, we will index $\F$ by a finite set $\Gamma,$ that is, $\calF = \{f_i\}_{i \in \Ga}.$ Sets of the form $f_{\om_1} \circ \dots  f_{\om_n}(K)$ are called \textit{level-n basic sets}. We call the map $\pi: \Ga^\N \rightarrow A$ defined by
$$\pi(\om) = \pi((\om_1, \om_2,\dots )) = \lim_{k \rightarrow \infty}f_{\om_k}\circ \dots  \circ f_{\om_1}(0)$$
the \textit{natural projection} from $\Ga^\N$ to $A.$ The map is surjective, but not necessarily injective. 

\indent We now define a natural class of measures on $A.$ Let $\mathbf{p} = \{p_i\}_{i \in \Ga} \in \mathcal{M}_1(\Gamma).$  We can then define the Borel probability measure $\pr$ on $\Gamma^\mathbb{N}$ coming from $\mathbf{p}$ via the product topology. The below is well known and can be found in \cite{fal}.
\begin{thm}\label{selfsimilar.thm}
    There is a unique Borel probability measure $\mu$ on $A$ with
    $$\mu = \sum_{i \in \Ga}p_if_i\mu.$$
\end{thm}
Measures defined using the above procedure are called \textit{self-similar}.
\subsection{Sketch of the proof}
We give a sketch proof. The reader is invited to picture the Sierpi\'nski carpet. The proof in general is identical. A Fourier analytic lemma below will give us a finite set of directions for which we will be able to find the prescribed covering of tubes. A fact, see Section \ref{sect.final} is that one can use horizontal, vertical, and diagonal (in both directions) tubes. Say we are at stage $n$ of the construction. In each of these directions, consider the projected IFS, defined below, and consider the level-$n$ basic sets for which the digit expansion in the projected IFS is away from being typical. Take the collection of preimages under the respective orthogonal projections. This will be an efficient cover: if we take a level-$n$ basic set in the Sierpi\'nski carpet, say, we will be able to show that its digit expansion in at least one of the prescribed directions will be away from being typical.   
\subsection*{Acknowledgements} The author would like to thank Andr\'as M\'ath\'e for all of his advice and suggestions. Extended thanks are given to Tim Austin and Tam\'as Keleti, for their suggestions which greatly improved the presentation. Further thanks are given to Attila G\'asp\'ar and Pablo Shmerkin for their advice and suggestions. Finally, thanks are given to the anonymous referee whose comments ensured a greater quality of exposition.

\section{Projections of $\times N$-invariant measures}
\begin{definition}
Define $M_n: \R \rightarrow [0,n)$ to be the mod $n$ map, which maps $x \in \R$ to the unique $0 \leq r < n$ which solves $x = nq + r$ for some $q \in \Z.$ If $n=1$ we refer to $M_1$ by $M.$
\end{definition}
Fix $K$ as in \eqref{eq.attractor}. Let $\calM(K,T)$ denote all the $T$-invariant measures supported on $K.$ The following is a basic fact about the space $\calM(K,T).$
\begin{lemma}[p97-p88, \cite{einsiedlerwardbook}]
    The space $\calM(K,T)$ is non-empty, and compact with respect to the weak* topology.
\end{lemma}
The following well known result says that measures on the torus are uniquely determined by their Fourier coefficients at integer frequencies. 
\begin{thm}[(3.66), \cite{mat2}]\label{fouriercircle.thm}
Let $\mu$ be a Borel measure supported on $[0,1]^d.$ Then $\mu = \Le^d_{|[0,1]^d}$ if and only if $\hat\mu(v) = 0$ for all $v \in \Z^d.$ Further, a sequence of probability measures $(\mu_n)_{n \in \N}$ in $[0,1]^d$ converges to $\mu \in \calM([0,1]^d)$ if and only if $\hat\mu_n(v) \rightarrow \hat\mu(v)$ for all $v \in \Z^d.$
\end{thm}
The following three lemmas are contained in Lemma 4.1 in \cite{pyo}. We include their simple proofs for the convenience of the reader. 
\begin{lemma}\label{coef.lemma}
    There exists an absolute constant $c > 0$ so that for all $\mu \in \calM(K,T)$ we can find a $v \in \Z^d\setminus\{0\}$ so that $|\hat\mu(v)| > c.$
    \begin{proof}
        Suppose the result is false. Then there exists a sequence of measures $(\mu_n)_{n \in \N}$ in $\calM(K,T)$ so that $|\mu_n(v)| \rightarrow 0$ for all $v \in \Z^d\setminus\{0\}.$ Therefore, by Theorem \ref{fouriercircle.thm}, we have that $\mu_n \rightarrow \calL^d_{|[0,1]^d}$ weak*. Since $\calM(K,T)$ is compact, by passing to subsequence we can assume that $(\mu_n)_{n \in \N}$ converges to $\mu \in \calM(K,T)$ weak*. But then $\calL^d_{|[0,1]^d} \in \calM(K,T),$ a contradiction, since $\calL^d(K) = 0.$
     \end{proof}
\end{lemma}
\begin{lemma}\label{eq}
For any $\mu \in \M(K,T)$  and any $v \in \Z^d \setminus \{0\}$ we have  $P_v\mu \ll \Le^1$ if and only if $MP_v\mu = \Leb.$ 
\begin{proof}
Fix $v \in \Z^d\setminus \{0\}$ suppose that $MP_v\mu:= P_v\mu \circ M^{-1} = \Leb.$ Let $N \subset \R$ be such that $\Le^1(N) = 0.$ Then 
$$P_v\mu(N) \leq P_v\mu(M^{-1}M(N)) = MP_v\mu(M(N)) = \Le^1(M(N)) = 0,$$
and so $P_v\mu \ll \calL^1.$

\indent Now suppose that $MP_v\mu \neq \Leb.$ By Theorem \ref{fouriercircle.thm} there exists $z \in \Z\setminus \{0\}$ such that $\widehat{MP_v\mu}(z) \neq 0.$ We then have, using the $T$-invariance of $\mu,$
\begin{align*}
\widehat{MP_v\mu}(z) &= \int e^{-2\pi i xz} dMP_v\mu(x) \\
&= \int e^{-2\pi i M(x)z}dP_v\mu(x) \\
&= \int e^{-2\pi i xz}dP_v\mu(x) \\
&= \int e^{-2\pi i x\cdot vz}d\mu(x) \\
&= \int e^{-2\pi i x\cdot vz}dT\mu(x) \\
&= \int e^{-2\pi i x\cdot Nvz}d\mu(x) \\
&= \int e^{-2\pi i xNz}dMP_v\mu(x) \\
&= \widehat{MP_v\mu}(Nz).
\end{align*}
Therefore by iterating we have that for all $k \in \N,$ $\widehat{MP_v\mu}(N^kz) = \widehat{MP_v\mu}(z).$ Then since  for each $a \in \Z$
$$\widehat{MP_v\mu}(a) =  \int e^{-2\pi i xa} dMP_v\mu(x) = \int e^{-2\pi i xa} dP_v\mu(x) = \widehat{P_v\mu}(a),$$
it follows that for all $k \in \N,$ $\widehat{P_v\mu}(N^kz) = \widehat{P_v\mu}(z)\neq 0$ and therefore $P_v\mu \not\ll \Le^1$ by the Riemann--Lebesgue lemma. 
\end{proof}
\end{lemma}
\begin{lemma}\label{fourier}
There exists a finite collection $\mathcal{V} \subset \mathbb{Z}^d\setminus\{0\}$ such that for every $\mu \in \M(K,T)$ there exists a $v \in \mathcal{V}$ such that $P_{v}\mu \not\ll \mathcal{L}^{1}.$ 
\end{lemma}
\begin{proof}
We first claim that there exists a finite $\calV \subset \Z^d\setminus\{0\}$ such that for any $\mu \in \calM (K,T)$ we may find a $v \in \V$ such that $\hat\mu(v) \neq 0.$ Suppose this is false. Then we can find a sequence $(\mu_n)_{n\in\N}$ in $\calM(K,T)$ and a sequence $(v_n)_{n \in \N}$ in $\Z^d\setminus\{0\}$ with $\hat\mu_n(v_n) \neq 0,$ so that $|v_n| \rightarrow \infty$ and each $v_n$ is chosen to be of minimal length so $\hat\mu(v_n) \neq 0.$ By passing to a subsequence we can assume that $\mu_n$ converges weak* to $\mu \in \calM(K,T).$ Now let $v \in \Z^d\setminus\{0\}$ so that $|\hat\mu(v)| > c,$ where $c$ is as in Lemma \ref{coef.lemma}. We know from Theorem \ref{fouriercircle.thm} that $\hat\mu_n(v) \rightarrow \hat\mu(v).$ Therefore for all $n$ large enough we have that $\hat\mu_n(v) \neq 0$ which contradicts the assumptions. 

\indent Now let $\mu \in \calM(K,T).$ For this set $\V$ we can find a $v \in \V$ such that $\hat\mu(v) \neq 0.$ By a simple observation we see that $\widehat{MP_v\mu(1)} \neq 0,$ and so by the argument in the final sentence of the previous lemma we see that $P_v\mu \not\ll \calL^1.$
\end{proof}
\section{The Sierpi\'nski carpet is tube-null}
Recall that we wish to prove that the attractor $K$ of 
$$\F = \Big\{f_i(x) = \frac{x}{N} + \frac{i}{N}\Big\}_{i \in \Gamma},$$ is tube-null, where $\Gamma$ is a proper subset of $\{0,\dots,N-1\}^d.$ We wish to look at the projection of this IFS in rational directions.
\begin{definition}
    For $v \in \Z^{d}\setminus\{0\}$ we define the \textit{projected IFS} of $\F$ in \textit{direction} $v$ by 
    $$\calF_v = \Big\{f^v_i(x) = \frac{x}{N} + \frac{i\cdot v}{N}\Big\}_{i \in \Gamma}.$$
    \end{definition}
    By setting $\Gamma_v = \Gamma\cdot v$ we may rewrite the above as 
        $$\calF_v = \Big\{f_i^v(x) = \frac{x}{N} + \frac{i}{N}\Big\}_{i \in \Gamma_v}.$$
        Define the map $\Pi_v: \Gamma \rightarrow \Gamma_v$ by $i \mapsto i\cdot v.$ For any $\text{p} \in \mathcal{M}_1(\Gamma_v)$ we can define the push-forward measure $M\text{p} \in \mathcal{M}_1(M(\Gamma_v)).$ For notational simplicity define 
$$\Sigma = \{0,1,\dots ,N-1\}.$$
Note that $M_N(\Ga_v)\subset \Sigma.$ Therefore we suppose that $M_N\text{p}\in \mathcal{M}_1(\Sigma)$ by the inclusion map and by setting $M_N\text{p}(i) = 0$ to any $i \in \Sigma$ where $i \not\in M(\Gamma_v) .$ So $M_N\text{p}$ is an $N$-tuple $(p_1,\dots ,p_N)$ where $M_N\text{p}(i/N) = p_{i+1}$ for $i = 0,\dots , N-1.$ 

\indent Consider the symbol space $\Ga^n_v.$ 
\begin{definition}
For $\om, \eta \in \Gamma^n_v$ we write $\om \sim \eta$ if $f^v_\om(0) = f^v_\eta(0).$ 
\end{definition}We can then redefine $\Ga^n_v$ by choosing, in a convenient manner, an element from each equivalence class. By doing this, we have removed the exact overlaps. Without loss of generality, by means of a translation, we assume that each element of $\Ga_v$ is positive.
\begin{lemma}\label{lem.equiv}
There exists $L \in \N$ so that for all $n \in \N$ and all  $\eta \in \Ga^n_v$ we may find $\om \in \{0,\dots,L\} \times \Sigma^{n-1}$ so that $\om \sim \eta. $
\end{lemma}
\begin{proof}
Set $L_1 = \max \Ga_v.$ For all $\eta \in \Ga^n_v$ we have 
\begin{align}
    f_\eta^v(0) &= \sum_{i=1}^n \eta_i / N^{n-i+1}\\
    &\leq L_1 \sum_{i=1}^\infty 1 / N^{i}\\
    &= L_1 \frac{1}{N-1}\\
    & \leq L_1.
\end{align}
 Set $L = L_1 N.$ Therefore, using the definition of the $f^v_j,$ for $\eta \in \Ga^n_v$ we have
\begin{equation}
    f_\eta^v(0) \in \{0, N^{-n},\dots, LN^{-n}\}.
\end{equation}
Recall that 
\begin{equation}\label{eq.funtionin}
    f_\eta^v(0) = f_{\eta_n}^v \circ \cdots \circ f^v_{\eta_1}(0),
\end{equation}
where $f_{\eta_j}(x) = x/N + \eta_j/N.$ We now choose to represent $\eta$ as follows: Find an integer $0 \leq j \leq 2L_1$ so that $f_\eta^v(0) \in [j,j+1).$ Set $\eta_1 = jN.$ We now know that 
\begin{equation}
    f_\eta^v(0) \in \{j, j + N^{-n}, \dots, j + 1 - N^{-n}\}.
\end{equation}
In the usual way we may find $\eta_2,\dots, \eta_n \in \Sigma$ so that 
\begin{equation}
        f_\eta^v(0) = f_{\eta_n}^v \circ \cdots \circ f^v_{\eta_1}(0).
\end{equation}
\end{proof} 
\indent The measures we will consider on $K$ will be the \text{self-similar} measures coming from probability vectors on $\Gamma$ as stated in Theorem \ref{selfsimilar.thm}. In fact, these measures are fixed points under an appropriate contraction.
\begin{thm}[Theorem 2.8, \cite{fal}]\label{selfsimunique.thm}
    Consider an IFS $\F = \{f_i\}_{i \in \Gamma}$ and place a probability vector $\p = \{p_i\}_{i \in \Gamma}$ on $\Gamma.$ Let $\mu$ be the unique self-similar measure coming from $\calF$ and $\mathrm{p}$ as defined in Theorem \ref{selfsimilar.thm}. Let $\calM$ be the class of Borel probability measures on $\R^d$ with bounded support. Endow $\calM$ with the metric $d,$ defined by
    $$d(\nu_1,\nu_2) = \sup\bigg\{\bigg| \int f d\nu_1 - \int f d\nu_2\bigg| : \mathrm{ Lip }~f \leq 1\bigg\},$$
    where $\mathrm{Lip}$ denotes the Lipschitz constant. Define the map $\varphi: \calM \rightarrow \calM$ by
    $$\varphi(\nu) = \sum_{i \in \Gamma} p_if_i \nu.$$ Then for any measure $\nu \in \calM$ we have $\varphi^n(\nu) \rightarrow \mu.$ That is, the measure $\mu$ is unique.
\end{thm}
\begin{lemma} \label{ent}
There exists a constant $\delta > 0,$ so that for any $\rm{p} \in \mathcal{M}_1(\Gamma)$ there exists a direction, $v \in \V,$ such that the measure $M\Pi_{v}\rm{p}$ on $\Sigma$ satisfies $\E(M\Pi_{v} \rm{p}) \leq 1 - \delta.$ 
\begin{proof}
Let $\mu$ be the self-similar measure coming from $\mathbf{p}.$ By Lemma \ref{fourier} there exists a $v \in \V$ such that $P_v\mu \not\ll \Le^1.$ Suppose that $\E(M\Pi_v\mathbf{p}) = 1.$ Therefore, by a basic property of Shannon entropy, we have $M\Pi_v\p = (1/N,\dots ,1/N).$ Write $\Pi_v\p = \{p_i\}_{i \in \Ga_v}.$ 
\begin{claim}
If $\nu$ is a probability measure on $\R$ with $M\nu = \Leb,$ then $M\varphi(\nu) = \Leb.$ 
\begin{proof}[Proof of claim]
Let $0 \leq k \leq N^n-1$ be an integer and consider the interval $[k/N^n, (k+1)/N^n).$ We then have 
\begin{align*}
M\varphi(\nu)\big([k/N^n, (k+1)/N^n)\big) &= M\bigg(\sum_{i \in \Ga_v}p_if_i\nu\bigg)\big([k/N^n, (k+1)/N^n)\big)\\
&= \sum_{i \in \Ga_v}p_i(M \circ f_i)\nu\big([k/N^n, (k+1)/N^n)\big)\\
&= \frac{1}{N}\sum_{i=0}^{N-1}\nu\big([k/N^{n-1} - i , (k+1)/N^{n-1}-i)\big)\\
&\leq \frac{1}{N} M\nu \big([k/N^{n-1}, (k+1)/N^{n-1})\big)\\
&= N^{-n}.
\end{align*}
Further if $M\varphi(\nu)[k/N^n, (k+1)/N^n)  < N^{-n}$ then there exists an integer $0 \leq l \leq N^{n-1}$ with $M\varphi(\nu)[l/N^n, (l+1)/N^n) > N^{-n}$ which contradicts the above. Therefore, we have 
$$M\varphi(\nu)[k/N^n, (k+1)/N^n)  = N^{-n}$$ 
and the claim follows from Hahn-Kolmogorov.
\end{proof}
\end{claim}

\indent Now let $\nu$ be a probability measure such that $M\nu = \Leb.$ (For example $\nu = \Leb.)$ We know by Theorem \ref{selfsimunique.thm} that $\varphi^k(\nu) \rightarrow P_v\mu,$ but since $M\varphi(\nu) = \Leb$ it follows that $MP_v\mu = \Leb$ by the continuity of $M.$ Therefore $P_v\mu \ll \Le^1$ which is a contradiction. Thus $\E(M\Pi_v\p) < 1.$  

\indent Now suppose there does not exist $\delta > 0$ as in the statement of Lemma \ref{ent}. Then there exists a sequence $\{\p_k\}_{k \in \N}$ in $\M_1(\Ga)$ with 
$$\lim_{k\rightarrow \infty}\E(M\Pi_v\p_k) = \E(M\Pi_v\lim_{k\rightarrow \infty}\p_k) \rightarrow 1$$ 
for all $v \in \V,$ with the equality following from the continuity of the maps $\E, M,$ and $\Pi_v.$ By compactness of $\M_1(\Ga)$ we can pass to a subsequence and find a $\p \in \ M_1(\Ga)$ such that $\p_k \rightarrow \p$ as $k \rightarrow \infty.$ Therefore $\E(M\Pi_v\p) = 1$ for all $v \in \V$ which contradicts the above.  
\end{proof} 
\end{lemma}
\indent We now pause to give a few more definitions and results we shall need before proceeding. Fix $v \in \V$ and $n \in \N.$
\begin{definition}[Definition 2.1.1, Definition 2.1.4, \cite{dem}]
For any $\om = (\om_1,\dots ,\om_n) \in \Gamma^n_v$ define  
$$L_n^\om = \frac{1}{n}\sum_{i=1}^n\delta_{\om_i}.$$
Note that $L_n^\om \in \M_1(\Gamma_v).$ Define the \textit{type class} of $\nu \in \M_1(\Ga)$ by 
$$T_n(\nu) = \{\om \in \Ga^n_v : L_n^\om = \nu\}.$$
Denote $\mathcal{L}_n$ the set of all possible types of sequences of length $n$ in $\Ga,$ i.e
$$\mathcal{L}_n = \{ \nu \in \M_1(\Ga) : \nu = L_n^\om \ \text{for some} \ \om \in \Ga^n_v\}.$$
\end{definition} 
\begin{lemma}[Lemma 2.1.2, \cite{dem}]\label{lem.upperboundent}
We have
\begin{align*}
|\mathcal{L}_n| \leq (n+1)^{|\Ga|}.
\end{align*}
\end{lemma}
\begin{lemma}[Lemma. 2.1.8, \cite{dem}]\label{lem.typeclassbound}
For every $\nu \in \mathcal{L}_n$ we have 
\begin{align*}
\frac{1}{(n+1)^{|\Gamma_v|}}e^{n\E(\nu)} \leq |T_n(\nu)| \leq 
e^{n\E(\nu)}.
\end{align*}
\end{lemma}
For the rest of this chapter, for two positive real numbers $x,y$ we say that $x \lesssim y$ if there exists a constant $C > 0$ that does not depend on $n$ so that $x \leq Cy.$
\begin{lemma}
Let $\delta$ be as in Lemma \ref{ent}. Define
$$A = \{ \nu \in \M_1(\Sigma) : \E(\nu) \leq 1-\delta\}.$$ 
$$S_n = \{\omega \in \Sigma^n : L_n^\om \in A \}.$$
Then 
$$|S_n| \lesssim n^{O(1)}N^{n(1-\delta)}.$$
\begin{proof}
We have
\begin{align*}
|S_n| &= \Big| \bigcup_{\nu \in \mathcal{L}_n}(S_n \cap T_n(\nu))\Big| \\
&= \sum_{\nu \in \mathcal{L}_n}|S_n \cap T_n(\nu)| \\
&= \sum_{\nu \in \mathcal{L}_n \cap A}|T_n(\nu)| \\
&\lesssim  \sum_{\nu \in \mathcal{L}_n \cap A}N^{nH(\nu)} \text{ by Lemma \ref{lem.typeclassbound}}\\
&\leq \max_{\nu \in \mathcal{L}_n \cap A}(n+1)^{|\Gamma_v|} N^{n\E(\nu)} \text{ by Lemma \ref{lem.upperboundent}}\\
&\leq (n+1)^{|\Gamma_v|}N^{n(1-\delta)}.
\end{align*}
\end{proof}	
\end{lemma}
\begin{lemma}\label{count}
Let $\delta$ be as Lemma \ref{ent}. Define
$$A = \{ \nu \in \M_1(\Sigma) : \E(\nu) \leq 1-\delta\}$$
and 
$$T_n^v = \{\om \in \Ga_v^n: ML_n^w \in A\}.$$
Then 
    $$|T_n| \lesssim n^{O(1)}N^{n(1-\delta)}.$$
\begin{proof}
Recall that 
$$\Ga_v^n \subset \{0,\dots M\} \times \Sigma^{n-1},$$
where $M$ is as in Lemma \ref{lem.equiv}.
Therefore 
$$M_N\Ga_v^n \subset M_N(m) \times \Sigma^{n-1} \subset \Sigma^n,$$
and so,
\begin{align*}
M_N(T_n^v) &= M_N(\{\om \in \Ga_v^n: ML_n^w \in A\})\\
&= \{M_N(\om) \in \Ga_v^n: ML_n^w \in A\}\\
&\subset \{\om \in \Sigma^n: L_n^w \in A\}\\
&= S_n.
\end{align*} 
Therefore 
$$T_n^v \subset M_M^{-1}(S_n),$$
and since the size of a pre-image of $M$ is at most a constant, the result follows. 
\end{proof}
\end{lemma}
\begin{proof}[Proof of Theorem \ref{main}]
Let $n \in \mathbb{N}.$ Partition the level $n$ basic sets of $K$ into sets $\{D^n_v\}_{v \in \V}$ as follows: Take a level $n$ basic set of $K;$ this is associated uniquely to some $\om \in \Gamma^{n}.$ Let $v \in \V$ be such that $\E(\Pi_{v}L_n^\om) \leq 1 - \delta.$ Place this basic set in $D^n_v.$ We can do this by Lemma \ref{ent}, and so $K \subset \bigcup_{v \in \mathcal{V}}D_v^n.$ We then have that $D_v^n \subset P_v^{-1}\pi_v(T_n).$ Then by Lemma \ref{count} $P_v^{-1}\pi_v(T_n)$ can be covered by $\sim N^{n(1-\delta)}$ $d-1$ dimensional hyperplanes of width $\sim N^{-n}.$ Now let $\epsilon >0.$ Choose $n \in \N$ large enough so that $
N^{-\delta} < \epsilon.$ The result then follows from Proposition \ref{hyper}.
\end{proof}
\section{Final remarks}\label{sect.final}
\subsection{Other tube-null sets}
We show that a class of sponge-like sets with a $N^{-1}$-adic grid structure are tube-null. We briefly recall the definition of a graph directed sets; a generalisation of iterated function systems. Let $\V$ be a set of $q$ vertices and let $\mathcal{E}$ be a collection of directed edges, so that $\mathcal{G} = (\V,\Ec)$ is a directed graph where for any two vertices there is a path of edges connecting them. For each $e \in \calE$ assign a contraction $f_e$ to it. Let $K_1,\dots ,K_q$ be the \textit{graph directed sets} associated to $\mathcal{G},$ that is, the unique non-empty compact sets $K_1,\dots ,K_q$ such that
\begin{equation}\label{eq.gds}
K_i = \bigcup_{j=1}^q\bigcup_{e \in \Ec_{i,j}}f_e(K_j).
\end{equation}
To see that this is a generalisation of an IFS, consider a set of contractions on $\R^d,$ say, $\calF.$ Let $\V = \{v\}$ be a single vertex and for each $f \in \F$ consider a directed edge $e,$ associated to $f,$ from $v$ to $v.$ For this graph we have
\begin{equation}
    K = \bigcup_{f \in \F} f(K) = \bigcup_{e \in \calE_{v,v}}f_e(K),
\end{equation}
and so \eqref{eq.gds} is satisfied. See \cite{fal} for a more in depth discussion on graph-directed sets. 
\begin{thm}\label{thm.general}
    Let $f:\R^d \rightarrow \R^d$ be a map of the form 
    $$f(x) = \frac{R(x)}{N} + \frac{i}{N},$$
    where $R$ is an isometry that maps the unit cube to itself, and $i \in \{0,1,\dots,N-1\}^{d}.$ Let $\F$ be a finite collection of maps of this form. Let $K_1,\dots,K_q$ be a collection of graph directed attractors derived from $\F.$ Then each $K_i$ is tube-null if and only if $\calL^d(K_i) = 0$ for each $i = 1,\dots,q.$
    \end{thm}

\begin{proof}
Let $G$ be the group of all isometries of $\R^d$ that map $[0,1]^d$ to itself. Consider the set
$$L = \bigcup_{i=1}^q\bigcup_{A \in G} A(K_i).$$ 
Since $G$ and $q$ are finite it is clear that $\Le^d(L) = 0$ if and only if each $K_i$ is Lebesgue-null. We will show that this set is invariant under the map $T.$ 

\indent As usual, let $\calE$ be the set of (directed) edges of the graph, $\calE_{i,j}$ be the set of directed edges from vertex $i$ to vertex $j,$ $f_e$ be the contraction associated to $e,$ and $A_e$ be the cube-fixing isometry associated to the contraction $f_e.$  By the definition of a graph directed set, we have, 
\begin{align*}
\bigcup_{i=1}^q\bigcup_{A \in G} A(K_i) &= \bigcup_{i=1}^q\bigcup_{A \in G} A\bigg( \bigcup_{j=1}^q\bigcup_{e \in \Ec_{i,j}}f_e(K_j))\bigg) \\
&= \bigcup_{i=1}^q\bigcup_{A \in G} A\bigg( \bigcup_{j=1}^q\bigcup_{e \in \Ec_{i,j}}A_e(K_j)/N + j_e/N)\bigg) \\
&=\bigcup_{i=1}^q\bigcup_{A \in G}  \bigcup_{j=1}^q\bigcup_{e \in \Ec_{i,j}}A\circ A_e(K_j)/N + A(j_e))/N. \\
\end{align*}
Then applying the map $T$ we have,
$$T(L) = \bigcup_{i=1}^q\bigcup_{A \in G}  \bigcup_{j=1}^q\bigcup_{e \in \Ec_{i,j}}A\circ A_e(K_j) = L.  $$
$\subset$ is straightforward to see. To see $\supset$ let $1 \leq k \leq q$ and $B \in G.$ We show that $B(K_k) \subset T(L).$ Let $e \in \calE_{i,k}.$ We see immediately that
\begin{equation}
    \bigcup_{A \in G}\bigcup_{e \in \calE_{i,k}}A\circ A_e(K_k) \subset T(L).
\end{equation}
Let $e \in \calE_{i,k}$ and by the transitivity of the group $G$ we may find $A \in G$ so that $A\circ A_e = B.$ Therefore $B(K_k) \subset T(L)$ as required. Since $B$ and $k$ were arbitrary, it follows that $L \subset T(L).$
Thus since $L$ is $T$-invariant which is a proper subset of $[0,1]^d$ it must be tube-null. Then clearly $K_i \subset L$ for each $i = 1,...,q$ and thus each $K_i$ is tube-null.   
\end{proof} 
\indent We are also able to generalise the result of Harangi \cite{har}. Consider the lattice of points, $\Lambda,$ on the plane defined by the vertices of a regular triangular lattice, centred at $0,$ with side-length $1/N.$
Let $T$ be the equilateral triangle in the plane of side 1 with vertices $(0,0), (1,0), (1/2, \sqrt 3/2).$ Let $A_1,...,A_6: \R^2 \rightarrow \R^2$ be rotation of $0, 60, 120, 180, 240, 300$ degrees respectively about the centre of $T.$ Define the IFS
$$\F = \{A_{k_i}(x)/N + j_i\}_{i \in \Ga}$$
where $\Ga$ is a finite indexing set, $k_i \in \{1,\cdots,6\}$ and each $j_i$ is an element of $\Lambda.$ Let $K$ be its attractor. We have the following result.
\begin{thm}\label{har}
The attractor $K$ is tube-null if and only if $\Le^2(K) = 0.$
\begin{proof}
Suppose that $\Le^2(K) = 0.$ Consider the union
$$\bigcup_{i = 1}^6 A_iK $$
Let $F: \R^2 \rightarrow \R^2$ be the affine map defined by  
$$F(x,y) = (x + y/2, \sqrt 3 y / 2).$$
Let $G$ be its inverse. Now define the set
$$L = \bigcup_{i=1}^6 G(A_i(K)).$$ 
We claim this map is invariant under the map $T.$ Indeed,
\begin{align*}
T(L) &= T\bigg( \bigcup_{j=1}^6\bigcup_{i\in \Ga} G(A_{k_i+j}(K))/N + G(j_i)\bigg)\\
&= \bigcup_{j=1}^6\bigcup_{i\in \Ga} G(A_{k_i+j}(K))\\
&= L
\end{align*}
Therefore $L$ is tube-null and since $K$ is a subset, so is $K.$
\end{proof}
\end{thm}
\begin{comment}
\subsection{Bounding the number of directions needed} 
The starting point for our proof of Theorem \ref{main} is Lemma \ref{ent}, with the directions needed coming from the abstract Lemma \ref{fourier}.
The proof of Lemma \ref{fourier} only doesn't give information on which directions are required and how large the set $\calV$ would have to be. We remark that for the Sierpi\'nski carpet, the horizontal, vertical, and the two diagonal directions all that is required. Further, we can show that taking just three of these will not be enough. It is reasonable that this could be done in general.
Let $K$ be the Sierpi\'nski carpet. Set $$\calV = \{(1,0),(0,1), (1,1), (1,-1)\}.$$
One can easily show that for just three directions from $\V$ we may find a probability vector on $\Gamma$ with $\E(M\Pi_v\rm{p}) = 1$ for all $v \in \V.$ The result then follows from Proposition \ref{prop.nottubenull}. 

\indent Now consider the four directions of $\V.$ It is an easy elementary-algebraic computation to show the following.
\begin{lemma}
    Let $\Gamma$ be the index for the IFS for $K.$ There exists $\delta >0$ so that for every $\rm{p} \in \calM_1(\Ga)$ there exists $v \in \V$ so that $\E(M\Pi_v\rm{p}) \leq 1-\delta.$
\end{lemma}
We omit the easy but ugly computation. The result then follows from this by following the rest of the proof the same way. It is reasonable that this could be done in general. 
\end{comment}
\bibliographystyle{alpha}
\bibliography{references.bib}

\end{document}